\documentclass[preprint]{elsarticle}

\usepackage{lineno,hyperref}
\usepackage{amssymb,euscript,amsmath, mathrsfs,amscd,enumitem}

\journal{Journal of Combinatorial Theory, Series A}

\newcounter{ENUM}
\newcommand{\be}{\begin{enumerate}}
\newcommand{\ee}{\end{enumerate}}
\newcommand{\beas}{\begin{eqnarray*}}
\newcommand{\eeas}{\end{eqnarray*}}
\newcommand{\bea}{\begin{eqnarray}}
\newcommand{\eea}{\end{eqnarray}}
\newcommand{\beq}{\begin{equation}}
\newcommand{\eeq}{\end{equation}}

\newcommand{\st}{\,:\,}

\newcommand{\pfn}{\mathrm{PF}_n}
\newcommand{\wipfn}{\mathrm{WIPF}_n}
\newcommand{\pfnrk}{\mathrm{PF}_n^{(r,k)}}
\newcommand{\wipfnrk}{\mathrm{WIPF}_n^{(r,k)}}

\newcommand{\pflad}{\tilde{F}^{(r)}_\lambda}

\newcommand{\fnr}{F_n^{(r)}}
\newcommand{\fnrk}{F_n^{(r,k)}}
\newcommand{\fnrx}{F_n^{(r,k)}}

\newcommand{\cprk}{\mathcal{P}^{(r,k)}(t)}
\newcommand{\cpr}{\mathcal{P}^{(r)}(t)}
\newcommand{\sn}{\mathfrak{S}_n}

\newtheorem{thm}{Theorem}[section]
\newtheorem{prop}[thm]{Proposition}
\newtheorem{lem}[thm]{Lemma}

\newdefinition{defn}[thm]{Definition}
\newdefinition{ques}[thm]{Question}
\newdefinition{ex}[thm]{Example}
\newdefinition{sit}[thm]{Situation}
\newdefinition{case}{Case}

\newdefinition{notn}[thm]{Notation}
\newdefinition{rem}[thm]{Remark}
\newdefinition{warn}[thm]{Warning}

\newproof{proof}{Proof}

\numberwithin{equation}{section}

\newcommand{\bm}[1]{{\boldsymbol{#1}}}
\def\arxiv#1{\href{https://arxiv.org/abs/#1}{arXiv:#1}}

\def\zz{\mathbb{Z}}

\def\pp{\mathbb{P}}
\newcommand{\qq}{\mathbb{Q}}

\newcommand{\bmu}{\bm{u}}

\begin{document}

\begin{frontmatter}

\title{Some Aspects of $(r,k)$-Parking Functions}


\author[mymainaddress,mysecondaryaddress]{Richard P. Stanley\fnref{myfootnote}}
\ead{rstan@math.mit.edu}
\fntext[myfootnote]{The first author was partially supported by NSF grant  DMS-1068625.}

\author[mysecondaryaddress]{Yinghui Wang\corref{mycorrespondingauthor}}
\cortext[mycorrespondingauthor]{Corresponding author}
\ead{yinghui@alum.mit.edu}

\address[mymainaddress]{Department of Mathematics, University of Miami, Coral Gables, FL 33124, United States}

\address[mysecondaryaddress]{Department of Mathematics, Massachusetts Institute of Technology, \\Cambridge, MA 02139, United States}

\begin{abstract}
An \emph{$(r,k)$-parking function} of length $n$ may be defined as a sequence $(a_1,\dots,a_n)$ of positive integers whose increasing rearrangement $b_1\leq\cdots\leq b_n$ satisfies $b_i\leq k+(i-1)r$. The case $r=k=1$ corresponds to ordinary parking functions. We develop numerous properties of $(r,k)$-parking functions. In particular, if $F_n^{(r,k)}$ denotes the Frobenius characteristic of the action of the symmetric group $\mathfrak{S}_n$ on the set of all $(r,k)$-parking functions of length $n$, then we find a combinatorial interpretation of the coefficients of the power series $\left( \sum_{n\geq 0}F_n^{(r,1)}t^n\right)^k$ for any $k\in \mathbb{Z}$. When $k>0$, this power series is just $\sum_{n\geq 0} F_n^{(r,k)} t^n$; when $k<0$, we obtain a dual to $(r,k)$-parking functions. We also give a $q$-analogue of this result. For fixed $r$, we can use the symmetric functions $F_n^{(r,1)}$ to define a multiplicative basis for the ring $\Lambda$ of symmetric functions. We investigate some of the properties of this basis.
\end{abstract}


\end{frontmatter}


\section{Introduction}\label{sec1}
Parking functions were first defined by Konheim and Weiss as
follows. Let $n$ be a fixed positive integer. We have $n$ cars $C_1,\dots,C_n$ and $n$ parking spaces
$1,2,\dots,n$. Each car $C_i$ has a preferred space $a_i$. The cars go
one at a time in order to their preferred space. If it is empty they
park there; otherwise they park at the next available space (in
increasing order). If all the cars are able to park, then the sequence
$\alpha=(a_1,\dots,a_n)$ is called a \emph{parking function} of length
$\ell(\alpha)=n$. For instance, $(3,1,4,3)$ is not a parking function
since the last car will go to space 3, but spaces 3 and 4 are already
occupied. It is easy to see that $(a_1,\dots,a_n)\in [n]^n$ (where
$[n]=\{1,2,\dots,n\}$) is a parking function if and only if its
increasing rearrangement $b_1\leq b_2\leq \cdots \leq b_n$ satisfies
$b_i\leq i$.

Let $\pfn$ denote the set of all parking functions of length $n$. A
fundamental result of Konheim and Weiss \cite{k-w} (earlier proved in
an equivalent form by Steck \cite{steck}---see Yan \cite[\textsection
  1.4]{yan} for a discussion) states that
$\#\pfn=(n+1)^{n-1}$. An elegant proof of this result was given by
Pollak (reported in \cite{riordan}), which we now sketch since it will
be generalized later. Suppose that we have the same $n$ cars, but now
there are $n+1$ spaces $1,2,\dots,n+1$. The spaces are arranged on a
circle. The cars follow the same algorithm as before, but once a car
reaches space $n+1$ and is unable to park, it can continue around the
circle to spaces $1,2,\dots$ until it can finally park. Of course all
the cars can park this way, so at the end there will be one empty
space. Note that their preferences $(a_1,\dots,a_n)\in[n+1]^n$ will be
a parking function if and only if the empty space is $n+1$. If the
empty space is $e$ and the preferences are changed to
$(a_1+i,\dots,a_n+i)$ for some $i$, where $a_j+i$ is taken modulo
$n+1$ so that $a_j+i\in[n+1]$, then the empty space becomes
$e+i$. Hence given $(a_1,\dots,a_n)\in [n+1]^n$, exactly one of the
vectors $(a_1+i,\dots,a_n+i)$ will be a parking function. It follows
that $\#\pfn=\frac{1}{n+1}(n+1)^n=(n+1)^{n-1}$.

We will use notation and terminology on symmetric functions from
\cite[Chap. 7]{ec2}.  The symmetric group $\sn$ acts on $\pfn$ by
permuting coordinates. Let $F_n \mathrel{\mathop:}=\mathrm{ch}\,\pfn$
denote the Frobenius characteristic of this action of $\sn$.
Recall from \cite[{\textsection}7.18, p.~351]{ec2} that 
the {\it Frobenius characteristic} $\mathcal{F}_{\varphi}$ of an action $\varphi$ of $\sn$ on a class of combinatorial
object $\mathcal{C}$ (e.g., $\pfn$) is
$$
\mathcal{F}_{\varphi}\ :=\ \frac{1}{\,n!\,}\sum_{w\in \sn} {\rm Fix}_{\varphi}(w) p_{\rho(w)},
$$
where ${\rm Fix}_{\varphi}(w)=\#\{c\in \mathcal{C}: {\varphi}(w).c=c\}$, 
$\rho(w)$ is the partition defined by the decreasingly
sorted lengths of cycles in $w$, and 
$p_{\lambda}$  is the power sum symmetric function indexed by partition ${\lambda}$.
Hence the Frobenius characteristic $F_n$ is a
homogeneous symmetric function of degree $n$, called the \emph{parking
  function symmetric function}. If $\alpha=(a_1,\dots,a_n)$ is a
sequence of positive integers with $m_i$ $i$'s (so $\sum m_i=n$), then
the Frobenius characteristic of the action of $\sn$ on the set of
permutations of the terms of $\alpha$ is the complete symmetric
function $h_{m_1}h_{m_2}\cdots$ (with $h_0=1$). Hence to compute
$F_n$, take all vectors $(b_1,\dots,b_n)\in \pfn$ with $b_1\leq
b_2\leq\cdots\leq b_n$ (the number of such vectors is the Catalan
number $C_n$) and add the corresponding $h_\lambda$ for each. For
instance, when $n=3$ the weakly increasing parking functions are 111,
112, 113, 122, 123, so $F_3= h_3+3h_2h_1+h_1^3$.

The symmetric function $F_n$ has many remarkable properties,
summarized (in a dual form, and with equation~\eqref{eq:fne} below not
included) in \cite[Exer.~7.48(f)]{ec2}. 

\begin{prop} \label{prop:fn}
We have
\begin{align}
 F_n\  = &\ \sum_{\lambda\vdash n} (n+1)^{\ell(\lambda)-1}
      z_\lambda^{-1}p_\lambda \nonumber\\  
      = &\
   \frac{1}{n+1}\sum_{\lambda\vdash n}s_\lambda(1^{n+1})s_\lambda   
    \nonumber  \\ 
    = &\ \frac{1}{n+1}\sum_{\lambda\vdash n}\left[ \prod_i
    \binom{\lambda_i+n}{\lambda_i}\right]m_\lambda \nonumber\\ 
    = &\
    \sum_{\lambda\vdash n}\frac{n(n-1)\cdots(n-\ell(\lambda)+2)}
    {d_1(\lambda)!\cdots d_n(\lambda)!}h_\lambda \nonumber 
    \\ 
   = &\ \sum_{\lambda\vdash n} \varepsilon_\lambda  
     \frac{(n+2)(n+3)\cdots (n+\ell(\lambda))}{d_1(\lambda)!
       \cdots d_n(\lambda)!}e_\lambda  \label{eq:fne} \\ 
   \omega F_n \ = &\ \frac{1}{n+1} \left[ \prod_i
  \binom{n+1}{\lambda_i}\right]m_\lambda \nonumber, 
\end{align}
where $d_i(\lambda)$ denotes the number of parts of $\lambda$ equal to
$i$ and $\varepsilon_\lambda=(-1)^{n-\ell(\lambda)}$.
Recall that 
$\omega$ is the involution defined by $\omega(e_n)=h_n$ for all $n\ge 1$, as in {\emph{\cite[\textsection7.6]{ec2}}}.  
Moreover,
  \beq F_n = \frac{1}{n+1}[t^n]H(t)^{n+1}, \label{eq:fnh2} \eeq
where $[t^n]f(t)$ denotes the coefficient of $t^n$ in the power series
$f(t)$, and
  $$ H(t) = \sum_{n\geq 0}h_nt^n= \frac{1}{(1-x_1t)(1-x_2t)\cdots}. $$
\end{prop}

Note in particular that the coefficient of $h_\lambda$ in
equation~\eqref{eq:fnh2} is the number of weakly increasing parking
functions of length $n$ whose entries occur with multiplicities
$\lambda_1, \lambda_2,\dots$. 

A further important property of $F_n$, an immediate consequence of
equation~\eqref{eq:fnh2} and the Lagrange inversion formula, is the
following. Let  
  \beq E(t) =\sum_{n\geq 0}e_nt^n=\prod_i (1+x_it), \label{eq:et}
  \eeq
and let $G(t)^{\langle -1\rangle}$ denote the compositional inverse of
the power series $G(t)$ (which will exist as a formal power series if
$G(t)=a_1t + a_2t^2+\cdots$, where $a_1\neq 0$). Then
$$ \sum_{n\geq 0}F_n t^{n+1}=(tE(-t))^{\langle -1\rangle}.$$

There are several known generalizations of parking functions. In
particular, if $\bmu=(u_1,\dots,u_n)$ is a weakly increasing sequence
of positive integers, then a $\bmu$-\emph{parking function} is a
sequence $(a_1,\dots,a_n)\in\pp^n$ (where $\pp=\{1,2,\dots\}$) such
that its increasing rearrangement $b_1\leq b_2\leq\cdots\leq b_n$
satisfies $b_i\leq u_i$. Thus an ordinary parking function corresponds
to $\bmu=(1,2,\dots,n)$. For the general theory of $\bmu$-parking
functions, see the survey \cite[{\textsection}13.4]{yan}. We will be
interested here in the special case
$\bmu=(k,r+k,2r+k,\dots,(n-1)r+k)$, where $r$ and $k$ are positive integers. We call such a
$\bmu$-parking function an $(r,k)$-\emph{parking function}. With this
terminology, an ordinary parking function is a $(1,1)$-parking
function. We call an $(r,1)$-parking function simply an
$r$-\emph{parking function}.

\textsc{Note.} Our terminology is not universally used. For instance,
if $(a_1,\dots,a_n)$ is what we call an $(r,r)$-parking function, then
Bergeron \cite{berg} would call $(a_1-1,\dots,a_n-1)$ an $r$-parking
function.

Pollak's proof that $\#\pfn=(n+1)^{n-1}$ extends easily to
$(r,k)$-parking functions. Namely, we now have $rn$ cars
$C_1,\dots,C_{rn}$ and $rn+k-1$ spaces $1,2,\dots,rn+k-1$. We consider
preferences $\alpha=(a_1,\dots,a_n)$, $1\leq a_i\leq rn+k-1$, where
cars $C_{r(i-1)+1},\dots,C_{ri}$ all prefer $a_i$. The cars use the
same parking algorithm as before. It is not hard to check that all the
cars can park if and only if $\alpha$ is an $(r,k)$-parking
function. Now arrange $rn+k$ spaces on a circle, allow the preferences
$1\leq a_i\leq rn+k$, and park as in Pollak's proof. Then $\alpha$ is
an $(r,k)$-parking function if and only if the space $rn+k$ is
empty. Reasoning as in Pollak's proof gives the following result,
which in an equivalent form is due to Steck \cite{steck}.

\begin{thm} \label{thm:genpol}
Let $\pfnrk$ denote the set of $(r,k)$-parking functions of length
$n$. Then
  $$ \#\pfnrk = k(rn+k)^{n-1}. $$
\end{thm}

The results in Proposition~\ref{prop:fn} can be extended to
$(r,k)$-parking functions (Theorem~\ref{thm:cprk}). For the case $k=r$, most of them
appear in Bergeron \cite[Prop.~1]{berg}. 
(Bergeron
and his collaborators have gone on to generalize their results in a
series of papers \cite{a-b, a-b2, berg2} on rectangular parking functions.) 
One of our
key results (Theorem~\ref{thm:relation}) connects $r$-parking
functions to $(r,k)$-parking functions as follows.

Let $\pfnrk$ denote the set of all $(r,k)$-parking functions of length
$n$, and let $\fnrk$ denote the Frobenius characteristic
$\mathrm{ch}\,\pfnrk$ of the action of $\sn$ on $\pfnrk$ by permuting
coordinates. 
Define
\begin{align*}
\cprk  &=  \sum_{n\geq 0} \fnrk t^n\\
         \cpr  &=  \mathcal{P}^{(r,1)}(t), 
\end{align*}  
then (Theorem~\ref{thm:relation})
  \beq \cpr^k=\cprk. \label{eq:cprk} \eeq
Equation~\eqref{eq:cprk} suggests looking at $\cpr^k$ for negative
integers $k$. We obtain parking function interpretations of the
coefficients of such power series in Section~\ref{sec:dual}.
As some motivation for what to expect, consider two power series 
$A(t),B(t)$, with $B(0)=0$, that are related by
 $$ A(t)=\frac{1}{1-B(t)}=1+B(t)+B(t)^2+\cdots. $$
Thus
  \beq B(t) = 1-\frac{1}{A(t)}, \label{eq:ba} \eeq
and often $B(t)$ will be a generating function for certain ``prime''
objects, while $A(t)$ will be a generating function for all objects,
i.e., products of primes. See for instance
\cite[Prop.~4.7.11]{ec1}. We will see examples of this relationship 
with our generating functions for parking functions.

For instance, if we set
  \beq  \mathcal{P}^{(r,k)}(t)^{-1} = 1-\sum_{n\geq 1}
   G_n^{(r,k)}t^n, \label{eq:P^{-1}} \eeq
then $G_n^{(1,1)}$ is the Frobenius characteristic of the action of
$\sn$ on \emph{prime} parking functions of length $n$, i.e., parking
functions that remain parking functions when some term equal to 1 is
deleted (a concept due to Gessel, private communication, 1997; see \cite[Exer.~5.49(f)]{ec2}). An
increasing parking function $b_1 b_2\cdots b_n$ can be uniquely
factored $\beta_1 \cdots \beta_k$, such that (1) if $b_j$ is the first
term of $\beta_i$ then $b_j=j$, and (2) if we subtract from each term
of $\beta_i$ one less than its first element (so it now begins with a
1), then we obtain a prime parking function.  

As a direct generalization of the previous example, $G_n^{(r,1)}$ is
the Frobenius characteristic of the action of $\sn$ on sequences $a_1
a_2\cdots a_n$ such that some $a_i=1$, and if remove this term then we
obtain an $(r,r)$-parking function.  More generally, if $1\leq k\leq
r$ then $G_n^{(r,k)}$ is the Frobenius characteristic of the action of
$\sn$ on sequences $a_1 a_2\cdots a_n$ such that we can remove some
term less than $k+1$ and obtain an $(r,r)$ parking function
(Theorem~\ref{thm:easyG}). For instance, when $r=2$ and $n=3$ the
increasing sequences with this property are 111, 112, 113, 114, 122,
123, 124, 222, 223, 224. Hence $G_3^{(2,2)}= 2h_1^3+6h_2h_1+2h_3$. The
situation for $\mathcal{P}^{(r,k)}(t)^{-j}$ when $j>r$ is more
complicated (Theorem~\ref{thm:G}).

\section{Expansions of $F_n^{(r,k)}$} \label{sec2}

In this section we consider the expansion of $F_n^{(r,k)}$ into the
six classical bases for symmetric functions.  These expresssions are
defined even when $k$ is an indeterminate, so we can use any of them
to define $\fnrk$ in this situation. For later combinatorial
applications we will only consider the case when $k$ is an integer.
We use notation from \cite[Ch.~7]{ec2} regarding symmetric
functions. We also use multinomial coefficient notation such as
 $$ \binom{k}{d_1,\dots,d_n,k-\sum d_i} = 
   \frac{k(k-1)\cdots(k-\sum d_i+1)}{d_1!\cdots d_n!}, $$
where $d_1,\dots,d_n$ are nonnegative integers and $k$ may be an
indeterminate. As usual we abbreviate $\binom{k}{d,k-d}$ as $\binom
kd$. 

\begin{thm} \label{thm:cprk}
Recall that $d_i(\lambda)$ denotes the number of parts of $\lambda$
equal to $i$. Then $F_0^{(r,k)}=1$, and for $n\geq 1$ we have
\begin{align}
F_n^{(r,k)}  = &\
 \frac{k}{rn+k} \sum_{\lambda \vdash  n} \binom{rn+k}{d_1(\lambda),
     \dots, d_n(\lambda),rn+k-\ell(\lambda)} h_{\lambda}
 \label{eq:hlambda}\\   
 = &\ 
   \frac{k}{rn+k}\sum_{\lambda\vdash n}\varepsilon_\lambda
    \binom{rn+k+\ell(\lambda)-1}{d_1(\lambda),\dots,d_n(\lambda),
      rn+k-1}e_\lambda \label{eq:elambda}\\ 
      = &\ 
 \frac{k}{rn+k} \sum_{\lambda \vdash  n} \left[ \prod_i
   \binom{\lambda_i+rn+k-1}{\lambda_i}\right] m_{\lambda} 
 \nonumber\\ 
 = &\
  \frac{k}{rn+k} \sum_{\lambda \vdash  n}
    s_{\lambda}(1^{rn+k})s_{\lambda} 
  \nonumber\\ 
  = &\
   k \sum_{\lambda \vdash  n} z_\lambda^{-1}(rn+k)^{\ell(\lambda)-1}
   p_{\lambda} \label{eq:plambda}\\ 
 \omega F^{(r,k)}_n
 = &\ 
  \frac{k}{rn+k} \sum_{\lambda \vdash  n} \left[ \prod_i
    \binom{rn+k}{\lambda_i}\right]  m_{\lambda}. \nonumber
\end{align} 
Moreover,
 \begin{equation}
    F^{(r,k)}_n = \frac{k}{rn+k} [t^n]H(t)^{rn+k}. \label{eq:[t^n]H}
\end{equation}
\end{thm}

\proof
Define two elements $\alpha$ and $\beta$
of $[rn+k]^n$ to be \emph{equivalent} if their difference is a multiple of
$(1,1,\dots,1)$ mod $rn+k$. This defines an equivalence relation on
$[rn+k]^n$, and each equivalence class contains $rn+k$ elements. It follows
from the proof of Theorem~\ref{thm:genpol} that each equivalence class
contains exactly $k$ $(r,k)$-parking functions. Moreover, all the
elements $\alpha$  in each equivalence class have the same multiset of part
multiplicities, i.e., the multiset $\{d_1,\dots,d_{rn+k}\}$, where
$d_i$ is the number of $i$'s in $\alpha$. 

For $n\geq 1$ let $D_n^{(r,k)}$ denote the Frobenius characteristic of
the action of $\sn$ on $[rn+k]^n$ by permuting coordinates. It follows
that
   $$ F_n^{(r,k)} =\frac{k}{rn+k}D_n^{(r,k)}. $$
Hence if we set $q=1,k=n$, and $n=rn+k$ in 
\cite[Exercise~7.75(a)]{ec2} then we get 
  $$ D_n^{(r,k)} = \sum_{\lambda\vdash
      n}s_\lambda(1^{rn+k})s_\lambda. $$
(Exercise 7.75 deals with $\mathfrak{S}_k$ acting on submultisets $M$
of $\{ 1^n,\dots,k^n\}$. Replace $M$ with the vector
$(d_1,\dots,d_k)$, where $d_i$ is the multiplicity of $i$ in $M$, to
get our formulation.) Therefore
  $$ F_n^{(r,k)} = \frac{k}{rn+k}\sum_{\lambda\vdash
      n}s_\lambda(1^{rn+k})s_\lambda. $$
The remainder of the proof is routine symmetric function
manipulation. 
\qed

\medskip

A further important property of $\fnrk$ in the case $k=r$, an immediate
consequence of equation \eqref{eq:[t^n]H} and the Lagrange inversion
formula \cite[Thm.~5.4.2]{ec2}, is the following.

Let $E(t)$ be given by equation~\eqref{eq:et}. Then
\begin{equation}\label{eq:inverse}
\sum_{n\ge 0} F^{(r,r)}_n t^{n+1} = 
  (tE(-t)^r)^{\langle-1\rangle} .
\end{equation}

\section{A relation between $r$-parking functions and $(r,k)$-parking  functions}
In this section we give a combinatorial proof of the following result.

\begin{thm} \label{thm:relation}
Let $k,r\in\pp$. Then $\cpr^k=\cprk$.
\end{thm}

\proof
Let $\wipfnrk$ denote the set of weakly increasing $(r,k)$-parking functions of length
$n$.
We need to give a bijection $\psi\colon (\wipfn^{(r,1)})^k\to
\wipfnrk$ such that if $\psi(\alpha_1,$ $\dots,\alpha_k)=\beta$, then
$\ell(\alpha_1)+\cdots+\ell(\alpha_k) = \ell(\beta)$.
Note that we consider the empty sequence $\emptyset$ to be an
$(r,j)$-parking function for any $r$ and $j$.

Given $(\alpha_1,\dots,\alpha_k)\in(\wipfn^{(r,1)})^k$, define
$\alpha'_i$ to be the sequence obtained by adding
$r(\ell(\alpha_1)+\cdots+\ell(\alpha_{i-1}))+i-1$ to every term of
$\alpha_i$, and let $\psi(\alpha_1,$ $\dots,\alpha_k)=(\alpha'_1,$ $\dots,\alpha'_k)$. For instance, if $r=2$ and
  $$ (\alpha_1,\dots,\alpha_5) = ((1,2),\emptyset,\emptyset,(1),
    (1,3,4)), $$
then $\alpha'_1=(1,2)$, $\alpha'_2=\alpha'_3=\emptyset$, $\alpha'_4=
(8)$, and $\alpha'_5=(11,13,14)$. 

 It is easily seen that $\psi$ is the desired bijection. In particular,
the inverse $\psi^{-1}$ has the following description. Given
$\beta=(b_1,\dots,b_n)\in\wipfnrk$, let $c_i=b_i-ri+r-1$. (The term
$r-1$ could be replaced by any constant independent from $i$; we made
the choice so $c_1=0$.) Let $c_{j_1}<\cdots<c_{j_r}$ be the
left-to-right maxima of the sequence $c_1,\dots,c_n$, so
$j_1=1$. Factor $\beta$ (regarded as a word $b_1\cdots b_n$) as
$\beta_1\cdots \beta_r$, where $\beta_i$ begins with
$b_{j_i}$. Subtract a constant $t_i$ from each term of $\beta_i$ so
that we obtain a sequence (or word) $\beta'_i$ beginning with a
1.\ \ Insert $c_{j_{i+1}}-c_{j_i}-1$ empty words $\emptyset$ between
$\beta'_i$ and $\beta'_{i+1}$, and place empty words at the end so
that there are $k$ words in all. These words $\alpha_1,\dots,\alpha_k$
then satisfy $\psi^{-1}(\beta)=(\alpha_1,\dots,\alpha_k)$.
\qed

\begin{ex}
Suppose that $r=2, k=7$, and
  $$ \beta=(1,2,2,10,12,14,15,19,22). $$
Then $(c_1,\dots,c_9)=(0,-1,-3,3,3,3,3,4,5)$. The left-to-right maxima
are $c_1=0$, $c_4=3$, $c_8=4$, $c_9=5$. Thus $\beta_1=(1,2,2)$,
$\beta_2=(10,12,14,15)$, $\beta_3=(19)$, and $\beta_4=(22)$. Hence
$\beta'_1=(1,2,2)$, $\beta'_2=(1,3,5,6)$,
$\beta'_3=\beta'_4=(1)$. Between $\beta'_1$ and $\beta'_2$ insert
$c_4-c_1-1=2$ copies of $\emptyset$. Similarly since
$c_8-c_4-1=c_9-c_8-1 =0$ we insert no further copies of $\emptyset$
between remaining $\beta'_i$'s. We now have the six words
$\beta'_1,\emptyset,\emptyset,\beta'_2,\beta'_3,\beta'_4$, Since $k=7$
we insert one $\emptyset$ at the end, finally obtaining
  $$ \psi^{-1}(\beta)=((1,2,2),\emptyset,\emptyset,(1,3,5,6),(1),(1),
      \emptyset). $$
\end{ex}

Theorem~\ref{thm:relation} has a natural $q$-analogue. We simply state
the relevant result since the bijection in the proof of
Theorem~\ref{thm:relation} is compatible with our $q$-analogue, so the
proof carries over. 

Given an $(r,k)$-parking function $\alpha=(a_1,\dots,a_n)$
of length $n$, note that the largest possible value of $\sum a_i$
is $k+(k+r)+\cdots+(k+(n-1)r)=kn+\binom n2 r$. Define
  \beq s^{(r,k)}(\alpha) = kn+\binom n2 r -\sum_{i=1}^n a_i. 
     \label{eq:srk} \eeq
     More specifically, using the
     notation of equation~\eqref{eq:srk} below it is easy to check  that if
     $\beta\in\pfnrk$ and $\psi^{-1}(\beta)=(\alpha_1,\dots,\alpha_k)$,
     then 
       $$ s^{(r,k)}(\beta)=\sum_{j=1}^k
          \left( s^{(r,1)}(\alpha_j)+(k-j)\ell(\alpha_j)\right). $$ 
When $k=r$ this is a well-known statistic on parking functions,
sometimes used in the variant form $\sum a_i$. See for instance
\cite{rs:shi}\cite[{\S\S}1.2.2,1.3.3]{yan}. Note that the action of
$\sn$ on $(r,k)$-parking functions $\alpha$ of length $n$ is
compatible with this statistic, i.e., if $w\in\sn$ then
$s^{(r,k)}(w\cdot \alpha)= s^{(r,k)}(\alpha)$.

Given a sequence $\beta=(b_1,\dots,b_n)\in\pp^n$, let $U_\beta$ denote
the Frobenius characteristic of the action by permuting coordinates of
$\sn$ on all permutations of the terms of $\beta$. Hence if $m_i$ is
the number of $i$'s in $\beta$ then $U_\beta= h_{m_1}h_{m_2}\cdots$.
Given $r,k,n\geq 1$, define
  $$ F_n^{(r,k)}(q) = \sum_\beta q^{s^{(r,k)}(\beta)} U_\beta, $$
where $\beta$ runs over all weakly increasing $(r,k)$-parking functions of
length $n$. Write
\begin{align*}
 \mathcal{P}^{(r,k)}(q,t) = &\ \sum_{n\geq
    0}F_n^{(r,k)}(q)t^n\\ 
   \mathcal{P}^{(r)}(q,t) = &\ \mathcal{P}^{(r,1)}(q,t).
\end{align*}
Thus $\mathcal{P}^{(r,k)}(1,t)=\mathcal{P}^{(r,k)}(t)$.

\begin{thm} \label{thm:q}
  We have
 $$ \mathcal{P}^{(r,k)}(q,t) = \prod_{i=0}^{k-1}
  \mathcal{P}^{(r)}(q,q^it). $$
\end{thm}

Equation \eqref{eq:ba} gives a relationship between a generating
function $A(t)$ for all objects and $B(t)$ for prime objects. There is
another basic relationship of this nature between exponential
generating funcions $A(t)$ for all objects and $B(t)$ for
``connected'' objects, namely, the \emph{exponential formula} $A(t)
=\exp B(t)$ or $B(t)=\log A(t)$. See \cite[{\S}5.1]{ec2}. Thus we can
ask whether there is a combinatorial interpretation of the
coefficients of $\log \cprk$. Recall that $D_n^{(r,k)}$ denotes the
Frobenius characteristic of the action of $\sn$ on $[rn+k]^n$ by
permuting coordinates, as in the proof of Theorem~\ref{thm:cprk}. The
case $k=r$ is handled by the following result.

\begin{prop}
We have 
  $$ \log \mathcal{P}^{(r,r)}(t)=\sum_{n\geq
     1}D_n^{(r,r)}\frac{t^n}{n}. $$
\end{prop} 

\proof 
The proof is a simple consequence of the following variant of
the Lagrange inversion formula appearing in \cite[Exer.~5.56]{ec2}:
for any power series $F(t)=a_1t+a_2t^2+\cdots\in\mathbb{C}[[t]]$ with
$a_1\neq 0$ we have
  \beq n[t^n] \log \frac{F^{\langle -1\rangle}(t)}{t} = [t^n] \left(
   \frac{t}{F(t)} \right)^n.  \label{eq:loglag} \eeq
Choose $F(t)=tE(-t)^r$, where $E(t)$ is given by
equation~\eqref{eq:et}. Now 
   $$ \frac{1}{E(-t)}= H(t)\:=\sum_{n\geq 0}h_nt^n. $$
Hence by equation~\eqref{eq:inverse}, we see that
equation~\eqref{eq:loglag} becomes
  $$ n[t^n]\log \mathcal{P}^{(r,r)}(t) = [t^n]
     H(t)^{nr}. $$
It is clear that $[t^n]H(t)^{nr}=D_n^{(r,r)}$, so the proof follows.
\qed

\section{A dual to $(r,k)$-parking functions} \label{sec:dual}
Equation~\eqref{eq:cprk} suggests looking at $\mathcal{P}^{(r)}(t)^k$
for negative integers $k$.  We obtain an object ``dual'' (in the sense
of combinatorial reciprocity) to $(r,k)$-parking functions.

We define $F^{(r,k)}_n$  for $k\le 0$ by \eqref{eq:hlambda}
(therefore all the equations in Theorem~\ref{thm:cprk}
hold for $k\le 0$).
It follows from the definition of $ \mathcal{P}^{(r,k)}(t)$ and
equation~\eqref{eq:cprk} that 
\begin{equation*}
\mathcal{P}^{(r)}(t)^k
= \mathcal{P}^{(r,k)}(t) 
= \sum_{n \ge 0} F^{(r,k)}_n t^n
\end{equation*}
holds for all $k>0$. Thus it also holds for all $k\le 0$.
Comparing the coefficients of $t^n$ with those in equation
\eqref{eq:P^{-1}}, namely, 
\begin{equation*}
\mathcal{P}^{(r)}(t)^{-k} 
= 1 - \sum_{n \ge 1} G^{(r,k)}_n t^n, \quad \mathrm{for\ all}\ \ k\ge 0,
\end{equation*}
and combining with \eqref{eq:hlambda},
we see that
\begin{equation*} 
G^{(r,k)}_n = - F^{(r,-k)}_n  = \frac{k}{rn-k}\sum_{\lambda \vdash  n}
{{rn-k}\choose {d_1(\lambda), \dots, d_n(\lambda)}} h_{\lambda},
\quad \mathrm{for\ all}\ k\ge 0,\ n\ge 1. 
\end{equation*}

We then have the following combinatorial interpretation of
$G^{(r,k)}_n$. 

\begin{thm}\label{thm:G}
If $n, r$ and $k$ are positive integers satisfying $rn-k>0$, then
$G^{(r,k)}_n$ is the Frobenius characteristic of the action of
$\mathfrak{S}_n$ on the set $S^{(r,k)}_n$ of $n$-tuples whose increasing
rearrangements have the following form: 
\begin{align}\label{eq:form}
\big(\underbrace{w, \dots, w}_{q(w)\ w\text{'s}},
b_{q(w)+1},b_{q(w)+2},\dots,b_{n} 
 \big),\end{align}
where $w\in [k]$ and $q(w)$ is the smallest integer such that $w\le
q(w)r$, and 
\begin{equation}\label{eq:b_j}  
b_{j} \le \min\{(j-1)r,w-1+rn-k\} \quad \text{for}\quad q(w)+1\le j\le n.
\end{equation} 
Note that we have $q(w)\le n$ as $w\le k<rn$.
Further, we see that $w\le \min\{(j-1)r,w-1+rn-k\}$ for all $j\ge q(w)+1$;
therefore \eqref{eq:b_j} is equivalent to 
\begin{equation}\label{eq:b_j'}
b_{j} \le \min\{(j-1)r,w-1+rn-k\} \quad \text{whenever}\quad b_{j}>w.
\end{equation}
In other words, a weakly increasing integer sequence (or equivalently, vector)  of length $n$
is in $S^{(r,k)}_n$ if
and only it satisfies the following properties: 
\begin{enumerate}[label={\rm{\bfseries  \Roman*.}}, leftmargin=*] 
\item  
$b_1=w$ for some $w\in [k]$, and $b_n-b_1< rn-k$;

\item  
$b_{q(w)}=w$;

\item
$b_{j} \le (j-1)r$ for all $j\in[n]$ whenever $ b_{j}>w$.
\end{enumerate}
\end{thm}

\begin{ex}
Let $r=1$, $k=2$, and $n=5$. Then $w\in\{1,2\}$ and $q(w)=w$. The coefficient of $t^5$ in
$-\mathcal{P}^{(1)}(t)^{-2}$ is
 $$ 2h_3h_1^2+2h_2^2h_1+4h_3h_2+4h_4h_1+2h_5. $$
This symmetric function is the Frobenius characteristic of the action
of $\mathfrak{S}_5$ on all sequences $(a_1,\dots,a_5)\in\pp^5$ whose 
increasing rearrangement $b_1\leq\cdots\leq b_5$ satisfies either of
the conditions 
(1) $w=q(w)=1$, $b_1=1$, $b_2\leq 1$ (so in fact $b_2=1$), $b_3\leq
2$, $b_4\leq 3$, $b_5\leq 3$, or 
(2) $w=q(w)=2$, $b_1=b_2=2$, $b_3\leq 2$ (so in
fact $b_3=2$), $b_4\leq 3$, $b_5\leq 4$. 
We get the fourteen
increasing sequences (orbit representatives) 11111, 11112, 11113,
11122, 11123, 11133, 11222, 11233, 11223, 22222, 22223, 22224, 22233,
22234. 
\end{ex}

\noindent {\bf A special case.}
When $k\in\{1,\dots,r\}$, 
for all $w\in[k]$ we have $q(w)=1$ and $(n-1)r\le rn-k \le w-1+rn-k$.
Therefore \eqref{eq:b_j'} becomes 
$b_j\le (j-1)r$ for all $j>1$, 
so $b$ having the form \eqref{eq:form} is equivalent to $b_1\in [k]$ and
$(b_2,b_3,\dots,b_n)$ is a weakly increasing $(r,r)$-parking functions
of length $n-1$. 
Thus Theorem \ref{thm:G} becomes the following result.

\begin{thm} \label{thm:easyG}
If $k\in\{1,\dots,r\}$, then
$G^{(r,k)}_n$ is the Frobenius characteristic of the action of
$\mathfrak{S}_n$ on the distinct $n$-tuples we get by adjoining $1, 2,
\dots$, or $k$ to $(r,r)$-parking functions of length $n-1$; 
or equivalently, 
the $n$-tuples whose increasing rearrangements start with $1, 2,
\dots$, or $k$ and followed by weakly increasing $(r,r)$-parking
functions of length $n-1$. 
\end{thm}

\smallskip
Theorem \ref{thm:G} is a consequence of the following key result, which
will be proved below Proposition \ref{prop:rotation2}. 

\begin{prop}\label{prop:rotation}
Suppose that $n, r$ and $k$ are positive integers such that $rn-k>0$. 
Given $a=(a_1, \dots, a_n) \in [rn-k]^n$, 
let $p\in [rn-k]$ be the smallest positive integer $i$ such that the
increasing rearrangements of $a$ and $(a+p\ {\rm {mod}}\ rn-k)$
coincide, where $a+i:=(a_1+i, \dots, a_n+i)$,
\begin{equation}\label{eq:mod}
a+i \ {\rm {mod}}\ rn-k:=(a_1+i \ {\rm {mod}}\ rn-k, \dots, a_n+i\ {\rm {mod}}\ rn-k),
\end{equation}
and \,$s\ {\rm {mod}}\ t$\, for integers $s$ and $t\,(>0)$ denotes the  $s$ taken modulo $t$ so that $s\ {\rm {mod}}\ t \in [t]$;
equivalently, $p=\#R_{rn-k}(a)$, where $R_{rn-k}(a)$ is the set  of increasing
rearrangements of vectors $(a+i\ {\rm {mod}}\ rn-k)$, $i\in
\mathbb{Z}$. 

Then the number of weakly increasing vectors $b\in S^{(r,k)}_n$ such that the
increasing rearrangement of $(b \ {\rm {mod}}\ rn-k)$ is in $R_{rn-k}(a)$ is
${pk}/({rn-k})$. 
\end{prop}

Theorem \ref{thm:G} follows as each $b\in S^{(r,k)}_n$ corresponds to a unique
set $R_{rn-k}(a)$ (the vector $a$ may not be unique). 

\begin{rem}
The reason why we need the vector $b \ {\rm {mod}}\ rn-k$ is that we
may have $b\in S^{(r,k)}_n\backslash [rn-k]^n$ and $b  \ {\rm {mod}}\ rn-k\in
S$. 
For instance, when $r=2,n=4,k=3$, $rn-k=5$, we have $(6,2,2,4)\in
S^{(r,k)}_n\backslash [rn-k]^n$ and  $(1,2,2,4)\in S^{(r,k)}_n$. 
\end{rem}

\noindent {\bf A special case.}
When $k\in\{1,\dots,r\}$,
it follows from \eqref{eq:b_j} that 
$b_n\le (n-1)r \le rn-k$ for all $b\in S^{(r,k)}_n$; therefore $b\ {\rm
  {mod}}\ rn-k = b$. 
In other words, we only need to consider $b$ instead of $b\ {\rm
  {mod}}\ rn-k$. 
Thus, combined with Theorem \ref{thm:G}, Proposition
\ref{prop:rotation} becomes as follows. 

\begin{prop}\label{prop:rotation2}
If $k\in\{1,\dots,r\}$, then for any given $(a_1, \dots, a_n) \in
[rn-k]^n$, there are exactly $k$ $i$'s $({\rm {mod}}\  rn-k)$ such
that the vector 
$a+i\ {\rm {mod}}\ rn-k$ defined by \eqref{eq:mod} is an $(r,r)$-parking
function of length $n-1$ adjoined by $1, 2, \dots$, or $k$, 
where $a_j +i\ {\rm {mod}}\ rn-k$ is the integer in $[rn-k]$ that is equal to $a_j +i$  modulo $rn-k$.
\end{prop}

\proof[Proposition \ref{prop:rotation}] 
We first clarify the structure of the set $R_{rn-k}(a)$ and the interpretations of $p$, then transform Case 1: $p<rn-k$ into Case 2: $p=rn-k$ (see Lemmas \ref{lem:a'} and \ref{lem:b} below), and finally deal with Case 2. 

For convenience, we denote $rn-k$ by $N$ and the
increasing rearrangement of a sequence $x$ by $x_{\uparrow}$,
then we have
$$R_{N}(a)=\{(a+i \ {\rm {mod}}\ N)_{\uparrow}:i\in \mathbb{Z}\},$$
and $p=\Pi_N(a)$ 
is the {\it smallest} positive integer such that 
\begin{equation}\label{eq:p}
(a+p \ {\rm {mod}}\ N)_{\uparrow}=(a \ {\rm {mod}}\ N)_{\uparrow}.
\end{equation}

Observe that the sequence $$T_{N}(a)\, :=\ \{(a+i \ {\rm {mod}}\ N)_{\uparrow}\}_{i=0}^{\infty}$$
is periodic with period $N$.
On the other hand,
by definition $p$ is the least period of $T_{N}(a)$. 
Therefore we have $p\, |\,N$. 

To see why we have $p=\#R_{N}(a)$, 
we notice that 
$\{(a+i \ {\rm {mod}}\ N)_{\uparrow}\}_{i=0}^{p-1}$ are $p$ pairwise distinct elements; otherwise, say $(a+i_1 \ {\rm {mod}}\ N)_{\uparrow}=(a+i_2 \ {\rm {mod}}\ N)_{\uparrow}$ for some $0\le i_1<i_2\le p-1$, then $T_{N}(a)$ has a period $i_2-i_1 <p$, which leads to a contradiction. 
Further, it follows from \eqref{eq:p} that $(a+(p+i) \ {\rm {mod}}\ N)_{\uparrow}=((a+p)+i \ {\rm {mod}}\ N)_{\uparrow}=(a+i \ {\rm {mod}}\ N)_{\uparrow}$ for all $i\in \mathbb{Z}$.
Hence $R_{N}(a)$ has exactly $p$ elements: $\{(a+i \ {\rm {mod}}\ N)_{\uparrow}\}_{i=0}^{p-1}$.

One can also interpret the elements of $R_{N}(a)$ as multisets of $n$ points on a circle as follows. 
Fix $N$ equally spaced points $Q_1,Q_2,\dots$, $Q_N$ clockwise around a circle centered at $\mathcal{O}$.
Map each vector $x=(x_1,x_2,\dots,x_n)\in [N]^n$
to a multiset $\mathcal{Q}(x):=\{Q_{x_1},Q_{x_2},\dots,Q_{x_n}\}$  of $n$ points (not necessarily distinct).
If the vector $x$ is further restricted to be weakly increasing, then this mapping is a bijection.
Notice that $\mathcal{Q}(x)=\mathcal{Q}(x_{\uparrow})$.
The point multiset $\mathcal{Q}({(x+i \ {\rm {mod}}\ N)_{\uparrow}})$ is then  obtained by rotating the point multiset $\mathcal{Q}({(x \ {\rm {mod}}\ N)_{\uparrow}})$
(all points with their multiplicities)  clockwise by $i\cdot 2\pi/N$ about the center $\mathcal{O}$.

Since it suffices to prove
the proposition for weakly increasing $a\in [N]^n$, 
from now on we assume that $a$ is weakly increasing. 
Applying the above rotation to vector $a$ and integer $p$, we see that the point multiset  $\mathcal{Q}(a)$ and its  rotation about  $\mathcal{O}$ by $p\cdot 2\pi/N$ clockwise, which is $\mathcal{Q}({(a+p \ {\rm {mod}}\ N)_{\uparrow}})$, are identical by \eqref{eq:p}. Therefore, 
the points in the multiset $\mathcal{Q}(a)$ are distributed identically on $N/p:=\ell$ arcs: $[Q_1, Q_{p}]$, $[Q_{p+1}, Q_{2p}]$, $\dots$, $[Q_{N-p+1}, Q_{N}]$.
It follows that $\mathcal{Q}(a)$ intersects each of these arcs at exactly $n/\ell:=n'$ points (not necessarily distinct) with the following form (in order of arcs): 
$\{Q_{a_1}, Q_{a_2},\dots,Q_{a_{n'}}\}$, $\{Q_{a_1+p}, Q_{a_2+p},\dots,Q_{a_{n'}+p}\}$, $\dots$, $\{Q_{a_1+(\ell-1)p}, Q_{a_2+(\ell-1)p},\dots,Q_{a_{n'}+(\ell-1)p}\}$.
In other words, we have shown the following result.

\begin{lem}\label{lem:a}
The vector $a$ has the form
\begin{align}\label{eq:d}
a=\ &\big({a_1}, {a_2},\dots,{a_{n'}},
{a_1+p}, {a_2+p},\dots,{a_{n'}+p},\\ \nonumber
& \ \ \dots, {a_1+(\ell-1)p}, {a_2+(\ell-1)p},\dots,{a_{n'}+(\ell-1)p}\big),
\end{align} 
or equivalently,
\begin{equation}\label{eq:+p}
a_{j+n'}=a_j+p \quad \text{for\ all}\quad j\in [n-n']
\end{equation}
with $1\le a_1\le a_2\le \cdots \le a_{n'}\le p$, where $\ell=N/p$ and ${n'}=n/\ell$.
\end{lem}

One may also deduce the form \eqref{eq:d} arithmetically from \eqref{eq:p}  by comparing the coordinates of both sides.

\smallskip
Further, we will need the following property of
$a':=({a_1}, {a_2},\dots,{a_{n'}})$ later 
to convert Case 1 -- the {\it periodic} case $p<N$ (i.e., $\ell\ge 2$) -- of Proposition \ref{prop:rotation}  to Case 2 -- the {\it non-periodic} case $p=N$ (i.e., $\ell=1$). 
Recall the definition of $\Pi_{\cdot}(\cdot)$ from \eqref{eq:p}.

\begin{lem}\label{lem:a'}
We have $\Pi_p({a'})=p$, where $a'=({a_1}, {a_2},\dots,{a_{n'}})$.
\end{lem}

\begin{proof}
Consider the multisets corresponding to the vectors  $a+i\ {\rm {mod}}\ N$ and $a'+i\ {\rm {mod}}\ p$ $(i\in \mathbb{Z})$, defined as follows. 
For $N\in \mathbb{Z}$ and $x=(x_1,x_2,\dots,x_n)\in \mathbb{Z}^n$, we denote the multiset
\begin{equation*}
M_{N}(x):= \{x_1,x_2 ,\dots,x_n\}\ {\rm {mod}}\ N :=
\{x_1 \ {\rm {mod}}\ N,x_2 \ {\rm {mod}}\ N,\dots,x_n \ {\rm {mod}}\ N\}.
\end{equation*}
Notice that  $M_{N}(x)=M_{N}(x_{\uparrow})$.
Thus it is equivalent to show that the multisets $M_{p}(a'+i)$, $i\in [p]$ are distinct.

Since $a$ has the form of \eqref{eq:d} and $N=\ell p$, we have
\begin{align}\nonumber
M_{N}(a+i)
&\ =\big\{a_{j_1}+i+j_2 \cdot p\ |\ j_1\in [n'],\ j_2\in [\ell]\big\}\ {\rm {mod}}\ N \\ \nonumber
&\ =\big\{(a_{j_1}+i)\ {\rm {mod}}\ p\, +j_2 \cdot p\ |\ j_1\in [n'],\ j_2\in [\ell]\big\}\ {\rm {mod}}\ N\\ \label{eq:M}
&\ =\big\{m +j_2 \cdot p\ |\ m\in M_{p}(a'+i),\ j_2\in [\ell]\big\}\ {\rm {mod}}\ N.
\end{align}
Here in the last equation, the multiplicities of $m$ are carried over to $M_{N}(a+i)$.

Recall that $\Pi_N(a)=p$; in other words, the multisets $M_{N}(a+i)$, $i\in [p]$ are distinct.
Hence the multisets $M_{p}(a'+i)$, $i\in [p]$ must be distinct. \qed
\end{proof}

Note that when $a$ is replaced by an arbitrary $a^*\in R_{N}(a)$, we have the same $p$, i.e.,  $\Pi_N(a^*)=\#R_{N}(a)$, and  
all the above arguments also work.

\medskip
Lemmas \ref{lem:a} and \ref{lem:a'} allow us to transform the {periodic} case to the {non-periodic} case through the following lemma.
Recall the definition of $S^{(r,k)}_n$ from Theorem \ref{thm:G}.

\begin{lem}\label{lem:b}
An $n$-dimensional vector $b$ is weakly increasing and satisfies 
$b\in S^{(r,k)}_n$ 
and $(b \ {\rm {mod}}\ N)_{\uparrow} \in R_{N}(a)$
if and only if $b$ has the form
\begin{align}\label{eq:b}
&\big({b'_1}, {b'_2},\dots,{b'_{n'}},
{b'_1+p}, {b'_2+p},\dots,{b'_{n'}+p},\\ \nonumber
& \ \ \dots, {b'_1+(\ell-1)p}, {b'_2+(\ell-1)p},\dots,{b'_{n'}+(\ell-1)p}\big)
\end{align} 
for some weakly increasing vector
$b'\in S^{(r,k')}_{n'}$ that satisfies 
$(b' \ {\rm {mod}}\ p)_{\uparrow} \in R_{N'}(a')$, where 
$$N=rn-k,\quad \ell=N/p,\quad {n'}=n/\ell,\quad k'=k/\ell,\quad a'=({a_1}, {a_2},\dots,{a_{n'}}),$$
and it follows that
$$N':=rn'-k'=(rn-k)/\ell=N/\ell=p>0.$$
Thus the assumption in Theorem \ref{thm:G} is fulfilled for $n'$, $r$ and $k'$, and $S^{(r,k')}_{n'}$ is well defined.
\end{lem}

Once the non-periodic case of Proposition \ref{prop:rotation} is proved,
the periodic case with given $n,r,k$ and $a$ will follow from Lemma \ref{lem:b}.
In fact, Lemma \ref{lem:b} implies that
the number of desired $b$'s in Proposition \ref{prop:rotation} is
the same as the number of desired $b'$'s in  Lemma \ref{lem:b}, and this number equals ${pk'}/({rn'-k'})=k'={pk}/({rn-k})$
by applying Lemma \ref{lem:a'} of non-periodicity to $a'$ and the non-periodic case of Proposition \ref{prop:rotation} to $n',r,k'$ and $a'$.
Hence Proposition \ref{prop:rotation} holds for the periodic case as well.

\begin{proof}[Proof of Lemma \ref{lem:b}]
{\it Sufficient condition.}
Suppose that $b$ has the form of \eqref{eq:b} for such a $b'$.
Then $b$  is weakly increasing because 
$b'$ is weakly increasing and  
\begin{equation}\label{eq:b'I}
b'_{n'}-b'_1<rn'-k'=p,
\end{equation}
as $b' \in S^{(r,k')}_{n'}$.

Next we show that $b\in S^{(r,k)}_n$ by verifying all the three properties in  Theorem \ref{thm:G}. 
We will use the properties in Theorem \ref{thm:G} for $b' \in S^{(r,k')}_{n'}$.

For Property I, we have $b_1=b_1' \in[k']\subseteq [k]$ by Property I of $b'$.
Further, it follows from \eqref{eq:b'I} that  
$$b_{n}-b_1 
=b'_{n'}-b'_1+(\ell-1)p
< 
p+(\ell-1)p=\ell p=N.$$

For Property II, observe that the $w$ is the same in cases $(n,r,k,b)$ and $(n',r',k',b')$;
therefore, by definition the $q(w)$ is the same in these two cases.
Since $w\le k'< rn'$, we see that $q(w)\le n'$ and hence 
$b_{q(w)}=b'_{q(w)}=w$ by Property II of $b'$. 

For Property III, for any $b_j>w$ with $j\in[n]$, let $d\ge 0$ and $j'\in [n']$ be integers such that $j=dn'+j'$.
Then we have the following two cases.

Case 1: $b'_{j'}>w$. 
Since $b'\in S^{(r,k')}_{n'}$, we have $b'_{j'} \le(j'-1)r$ by Property III. Thus 
\begin{equation*}
b_{j}= b'_{j'}+dp 
\le (j'-1)r+d(rn'-k')
< (j'-1+dn')r=(j-1)r.
\end{equation*}

Case 2: $ b_{j'}\le w<b_j$. Thus $d\ge 1$ and it follows that
\begin{equation*}
b_{j}= b'_{j'}+dp 
\le w+d(rn'-k')
\le k'-dk'+drn'
\le dn'r
\le (j-1)r.
\end{equation*}

Hence $b\in S^{(r,k)}_n$.

\smallskip
Now we show that $(b \ {\rm {mod}}\ N)_{\uparrow} \in R_{N}(a)$, in other words,
\begin{equation} \label{eq:b+i}
(b+i \ {\rm {mod}}\ N)_{\uparrow}=a,
\end{equation}
i.e.,
\begin{equation} \label{eq:i}
M_N(b+i)=M_N(a)
\end{equation}
for some $i\in \mathbb{Z}$.
Since $(b' \ {\rm {mod}}\ p)_{\uparrow} \in R_{p}(a')$, there exists $i\in \mathbb{Z}$ such that 
\begin{equation} \label{eq:i'}
M_p(b'+i)=M_p(a').
\end{equation}
Combining with \eqref{eq:M} for $b+i$ and for $a$ leads to
\eqref{eq:i}, as desired.

\smallskip
{\it Necessary condition.}
Suppose that $b$ is a weakly increasing  vector such that
$b\in S^{(r,k)}_n$ 
and $(b \ {\rm {mod}}\ N)_{\uparrow} \in R_{N}(a)$.
Then  there exists $i\in \mathbb{Z}$ that satisfies
\eqref{eq:b+i} and \eqref{eq:i}.

First we show that $b$ satisfies \eqref{eq:+p}.
Since $b\in S^{(r,k)}_n$, we have $b_n-b_1<N$.
Thus there exists $h\in [n]$ and $d\in \mathbb{Z}$ such that 
\begin{align*} 
(d-1)N <&\ b_1+i \le \cdots \le  b_h+i 
\le dN <b_{h+1}+i \le \cdots \le b_n+i \\
<&\ b_1+i+N \le  (d+1)N.
\end{align*}
Combining with \eqref{eq:b+i} yields
\begin{equation*} 
a = (b+i \ {\rm {mod}}\ N)_{\uparrow}
= \big( b_{h+1}+i,\dots,b_n+i,
b_1+i+N, \dots,  b_h+i+N \big) - dN,
\end{equation*}
i.e., 
\begin{equation*} 
b_j=
\begin{cases} 
a_{j-h}-i+dN, & {\mathrm{if}}\ j\in [h+1,n] \\ \nonumber
a_{j-h+n}-i+(d-1)N, & {\mathrm{if}}\ j\in [h]
\end{cases}.
\end{equation*}

For any given $j\in [n-n']$, we have the following three cases.

\smallskip
Case 1: $[j,j+n']\subseteq [h]$, 
i.e., $1\le j\in h-n'$. Then
$$b_{j+n'} -b_j = (a_{j+n'-h}-i+dN)-(a_{j-h}-i+dN)
=a_{j+n'-h}-a_{j-h}=p$$
by applying \eqref{eq:+p} to $j-h$.

\smallskip
Case 2: $[j,j+n']\subseteq [h+1,n]$, 
i.e., $h+1\le j\le n-n'$. Then
\begin{align*}
b_{j+n'} -b_j =&\ \big(a_{j+n'-h+n}-i+(d-1)N\big) - \big(a_{j-h+n}-i+(d-1)N\big)\\
=&\ a_{j+n'-h+n}-a_{j-h+n}=p
\end{align*}
by applying \eqref{eq:+p} to $j-h+n$.

\smallskip
Case 3: $j\subseteq [h]$ and $j+n'\subseteq [h+1,n]$, 
i.e., $h+1-n' \le j\le h$. Then
\begin{align*}
b_{j+n'} -b_j =&\ (a_{j+n'-h}-i+dN) - \big(a_{j-h+n}-i+(d-1)N\big)\\
=&\ a_{j+n'-h}-a_{j-h+n}+N
=-(\ell-1)p+\ell p = p
\end{align*}
by applying \eqref{eq:+p} to $j+n'-h,j+2n'-h,\dots,j+(\ell-1)n'-h$.

\smallskip
Hence  $b$ satisfies \eqref{eq:+p} and has the form of \eqref{eq:b} for some weakly increasing integer vector $b'$.

Further, applying \eqref{eq:i} and \eqref{eq:M} to $b+i$ and to $a$ leads to \eqref{eq:i'} and therefore
$(b' \ {\rm {mod}}\ p)_{\uparrow} \in R_{p}(a')$.

It remains to verifying all the three properties in  Theorem \ref{thm:G} for $(n',r,k',b')$ to obtain $b'\in S^{(r,k')}_{n'}$.
We will use the properties in Theorem \ref{thm:G} of $b \in S^{(r,k)}_{n}$.

For Property I, notice that $b_{n'+1}=b_1+p>b_1$. 
Applying Property III to $b_{n'+1}$, we get $b_{n'+1}\le rn'$. 
It follows that
$$b'_1=b_{n'+1}-p\le rn'-p = k'.$$
Further, applying the second part of Property I to $b$ yields
$$b'_{n'}-b'_1 
=b_{n'}-b_1-(\ell-1)p
<
N-(\ell-1)p=p=rn'-k'.$$

For Property II, similarly to that in the sufficient condition, we have $q(w)\le n'$ and hence 
$b'_{q(w)}=b_{q(w)}=w$ by Property II  of $b$.

For Property III, for any $b'_j>w$ with $j\in[n']\subseteq[n]$, we have $b_{j}=b'_{j}>w$. 
Therefore $b'_{j}=b_{j}\le (j-1)r$ by Property III  of $b$.
\qed
\end{proof}

Now let us show Proposition \ref{prop:rotation} for the non-periodic case $p=N$.

\smallskip

\noindent
{\it Proof of the non-periodic case of Proposition \ref{prop:rotation}.}
In this case, the vectors $(a+i\ {\rm {mod}}\ N)_{\uparrow}$, $i\in [N]$ are distinct.
To ease the notation, we write $S$ for $S^{(r,k)}_{n}$.
We will determine explicitly the  ${pk}/({rn-k})=k$ vectors in $S$
desired in Proposition \ref{prop:rotation}. 

For convenience, we denote $x_j=a_{j+1}\,(\le N)$, $j=0,\dots,n-1$,
and consider the weakly increasing vector $x=(x_0,\dots,x_{n-1})$
with $x_0=1$. 
Then $x\in S$ if and only if 
$x_j\le rj$ for all $j\in [n-1]$.
In general, a weakly increasing integer vector $y$ is in $S$ if and only if
\begin{enumerate}[label={\bfseries  \Roman*$^\prime$.}, leftmargin=*] 
\item  
$y_0=w$ for some $w\in [k]$, and $y_{n-1}-y_0< N$;

\item  
$y_{q(w)-1}=w$;

\item
$y_{j} \le jr$ for all $j\in [n-1]$ whenever $ y_{j}>w$.
\end{enumerate}
In the rest of the proof, all variables are integers, and for a vector
$y$, we denote by $y_j$ its $(j+1)$-th coordinate. 

\smallskip
Let $\Delta_j:=rj-x_j$, $j=0,1,\dots,n-1$.
Then $\Delta_0=-1$, and
$x\in S$ if and only if 
$\Delta_j\ge 0$ for all $j\in [n-1]$.

\begin{lem}\label{lemma:i_0}
There exists $i\in \mathbb{Z}$ such that the vector
$(x+i\ {\rm {mod}}\ N)_{\uparrow}\in S$, with the smallest coordinate
equal to $1$. More precisely, if $x\in S$, then we can take $i=0$;
otherwise, take 
$i=1-x_j$, where $j$ is the largest number in $[n-1]$ such that
$\Delta_j=\min_{j'\in [n-1]} \Delta_{j'}$.
\end{lem}

\proof[Lemma \ref{lemma:i_0}]
Assume that $x\notin S$, then $\Delta_j\le -1$ and $j\in [n-1]$ for
the $j$ taken in the lemma. Taking $i=1-x_j$, we get
\begin{align*} 
x+i\ {\rm {mod}}\ N
=&\ \big({2-x_j+N,x_1-x_j+1+N, \dots, x_{j-1}-x_j+1+N}, \\
&\ \  {1, x_{j+1}-x_j+1,\dots, x_{n-1}-x_j+1} \big),
\end{align*}  
and thus
\begin{align*} 
\alpha:=(x+i\ {\rm {mod}}\ N)_{\uparrow}
=&\ \big({1}, \underbrace{x_{j+1}-x_j+1}_{\alpha_1},\dots, \underbrace{x_{n-1}-x_j+1}_{\alpha_{n-1-j}},\\
&\ \  \underbrace{2-x_j+N}_{\alpha_{n-j}},
\underbrace{x_1-x_j+1+N}_{\alpha_{n-j+1}}, \dots, \underbrace{x_{j-1}-x_j+1+N}_{\alpha_{n-1}}
 \big).
\end{align*}

It follows from the definition of $j$ that
$x_j\ge rj+1$, and
for $j'>j$ we have
$\Delta_{j'}\ge \Delta_{j}+1$, and therefore
$x_{j'}-x_{j}\le r(j'-j)-1$;
for $j'<j$ we have
$\Delta_{j'}\ge \Delta_{j}$, and therefore
$x_{j'}-x_{j}\le r(j'-j)$.
Thus
\begin{align*} 
\alpha_{u}=&\  x_{j+u}-x_j+1\le  r(j+u-j)-1+1=ru,\quad u\in [n-1-j]\\
\alpha_{n-j}=&\ 2-x_j+rn-k\le 2-rj-1+rn-1=r(n-j)\\
\alpha_{n-j+u}=&\  x_u+1-x_j+rn-k\le r(u-j+n),\quad u\in [j-1].
\end{align*}  
Hence $\alpha\in S$.
\qed

\medskip
On the strength of Lemma \ref{lemma:i_0}, 
we can assume that $x \in S$ with $x_0=1$.
The following result determines the $k$ vectors in $S$ desired in
Proposition \ref{prop:rotation}. 

\begin{lem}\label{lemma:i}
Let $0=j_0<j_1<j_2<\cdots$ be the elements of the subset 
\begin{align*}  
J^*:= \{j\in J: \Delta_{j'} > \Delta_j,\ \mathrm{for\ all}\ n-1\ge
j'>j\}
\end{align*} 
of
\begin{align*}  
J:=\{0\} \cup \{j\in [n-1]:x_j>x_{j-1}\} 
\end{align*} 
and $m$ be the nonnegative integer determined by 
\begin{align*}  -1=\Delta_{j_0}<\Delta_{j_1}<\cdots<\Delta_{j_m}\le
  k-2<\Delta_{j_{m+1}}<\cdots 
\end{align*} 
(if $j_{m+1}$ does not exist, then set $j_{m+1}$ and
$\Delta_{j_{m+1}}$ to be infinity). 
In particular, $j_1$ is the largest number in $[n-1]$ such that
$\Delta_{j_1}=\min_{j\in [n-1]} \Delta_{j}\ge 0$. 

\smallskip
Then $y$ is a weakly increasing vector in $S$ such that  $(y \ {\rm
  {mod}}\ N)_{\uparrow} \in R_x$ if and only if  
\begin{enumerate}[label={\rm (\arabic*)}, leftmargin=*] 
\item 
$y=x+i$ with $0\le i\le \Delta_{j_1} \wedge (k-1)$,
where $\wedge$ represents the minimum function; or 

\smallskip
\item
$y=(x+i_1\ {\rm  {mod}}\ N)_{\uparrow} +i_2$ with
\begin{enumerate}[label={\rm (\roman*)}, leftmargin=0in] 
\item  
$i_1=1-x_{j_v}$ for some $v\in [m]$, and 

\item  
$0\le i_2=y_0-1\le \Delta_{j_{v+1}} \wedge (k-1)-\Delta_{j_v}-1<k-1$.
\end{enumerate}
\end{enumerate}

Further, the $k$ vectors  given in (1) and (2) are distinct.
\end{lem}

\begin{rem}
Note that (1) is the special case of (2) with $i_1=0=v$ and $i_2=i$. 
\end{rem}

\proof[Lemma \ref{lemma:i}]
As a consequence of $p=N$, the vectors $(x+i_1\ {\rm
  {mod}}\ N)_{\uparrow}$ with $i_1$ given in (1) ($i_1=i$) and (2),
whose smallest coordinates are all $1$, are distinct. Thus the $k$
vectors  given in (1) and (2) are distinct. 

\smallskip
(1) If $y=x+i\in S$, then by definition we have
$1\le (x+i)_0\le k$ and $(x+i)_{j_1} \le rj_1$. 
Thus $0\le i\le \Delta_{j_1} \wedge (k-1)$.

Conversely, for any $y=x+i$ with $0\le i\le \Delta_{j_1} \wedge (k-1)$,
we have  $(y \ {\rm {mod}}\ N)_{\uparrow} \in R_x$,
$y_{n-1}-y_0=x_{n-1}-x_0<N$, and $1\le w:=y_0=(x+i)_0\le
(1+\Delta_{j_1}) \wedge k\le k$, and Property I$^\prime$ follows. 

For Property II$^\prime$, 
notice that for any $j\in [n-1]$ such that $x_j\ge 2$, since
$rj-x_j=\Delta_j\ge \Delta_{j_1}$,  
we have $j\ge (2+\Delta_{j_1})/r>w/r$, and therefore $j\ge q(w)$. 
Hence $y_{q(w)-1}=w$.

Finally for Property III$^\prime$,  for all $j\in [n-1]$, since
$\Delta_{j_1}\le  \Delta_j$, we get 
$x_j-x_{j_1}\le r(j-j_1)$, 
and therefore 
$y_j=(x+i)_j=x_j+i\le r(j-j_1)+\Delta_{j_1}=rj$.

\smallskip
(2) 
If $y$ is a weakly increasing vector in $S$ such that $(y \ {\rm
  {mod}}\ N)_{\uparrow}\in R_x$ but $y$ does not have the form
described in (1), 
then by Lemma \ref{lemma:x_1=1} below, we get $\alpha:=y-i_2\in S$,
$\alpha_0=1$ and $\alpha\in R_x$, where $i_2=y_0-1\ge 0$. 

\begin{lem}\label{lemma:x_1=1}
If $b\in S$ is weakly increasing, then $b+i\in S$ for all $i\in \{0,-1,\dots,-b_1+1\}$.
Further, if $(b \ {\rm {mod}}\ N)_{\uparrow}\in R_{N}(a)$, then $(b+i
\ {\rm {mod}}\ N)_{\uparrow}\in R_{N}(a)$. 

In particular, when $i=-b_1+1$, the smallest coordinate of $b+i$ is $1$. 
Thanks to \eqref{eq:b_j'}, we have $b+i\in [N]^n$, and
therefore $b+i \ {\rm {mod}}\ N=b+i$. 
If $(b \ {\rm {mod}}\ N)_{\uparrow}\in R_{N}(a)$, then
$(b+i)_{\uparrow}=(b+i \ {\rm {mod}}\ N)_{\uparrow}\in R_{N}(a)$. 
\end{lem}

Lemma \ref{lemma:x_1=1} follows immediately from the definition of $S$ and $R_{N}(a)$. 

Since $\alpha\neq x$, we have 
$\alpha=(x+i_1\ {\rm {mod}}\ N)_{\uparrow}$ for some $i_1\in \{-1,-2,\dots,1-N\}$.
Recall that $\alpha_0=1$, and thus $i_1=1-x_j$ for some $j\in [n-1]$.
If there is more than one $j$ such that $i_1=1-x_j$, we choose the smallest one, i.e., the $j\in J$.  Then
\begin{align*} 
x+i_1\ {\rm {mod}}\ N
=&\ \big({2-x_j+N,x_1-x_j+1+N, \dots, x_{j-1}-x_j+1+N}, \\
&\ \ {1, x_{j+1}-x_j+1,\dots, x_{n-1}-x_j+1} \big),
\end{align*}  
and
\begin{align*} 
\alpha:=(x+i_1\ {\rm {mod}}\ N)_{\uparrow}
=&\ \big({1}, \underbrace{x_{j+1}-x_j+1}_{\alpha_1},\dots, \underbrace{x_{n-1}-x_j+1}_{\alpha_{n-1-j}},\\
&\ \ \underbrace{2-x_j+N}_{\alpha_{n-j}},
\underbrace{x_1-x_j+1+N}_{\alpha_{n-j+1}}, \dots, \underbrace{x_{j-1}-x_j+1+N}_{\alpha_{n-1}}
 \!\big).
\end{align*}   

Recall that $\alpha\in S$ if and only if
\begin{align} \label{eq:ru}
\alpha_{u}\le ru,\quad \ \mathrm{for\ all}\ \ u\in [n-1].
\end{align}   
Applying to $u=1,\dots,n-1-j$ leads to
$$x_{j'}-x_j+1\le r(j'-j),\ \text{i.e.,}\
\Delta_j<\Delta_{j'},\quad \mathrm{for\ all}\  j'<j\le n-1;$$
applying to $u=n-j$ leads to
$$2-x_j+rn-k\le r(n-j),\ \text{i.e.,}\
\Delta_j\le k-2. $$
Therefore $j=j_v$ for some $v\in[m]$.

Conversely, from the above argument we see that if $i_1=1-x_{j_v}$ for some $v\in[m]$, then we have $\alpha_{u}\le ru$ for all
$u\in [n-j]$. 
Further, we have $$\alpha_{n-j+u}=x_u-x_j+1+N \le
ru+\Delta_j-rj+1+rn-k<r(n-j+u)$$ for all $u\in[j-1]$. Hence $\alpha\in
S$. 

\smallskip
It remains to show that $\alpha+i_2\in S$ if only if $i_2$ satisfies
the inequality in (ii). 

If $\alpha+i_2\in S$, then applying \eqref{eq:ru} to
$\alpha':=\alpha+i_2$ and $u=j_{v+1}-j_v$ (if exists) leads to 
$$x_{j_{v+1}}-x_{j_v}+1+i_2\le r(j_{v+1}-j_v),\quad \text{i.e.,}\quad
i_2\le \Delta_{j_{v+1}}-\Delta_{j_v}-1;$$
applying \eqref{eq:ru} to $\alpha':=\alpha+i_2$ and $u=n-{j_v}$ leads to
$$2-x_{j_v}+rn-k+i_2\le r(n-{j_v}),\ \text{i.e.,}\
i_2\le  k-2-\Delta_{j_v}. $$
Recall that $i_2=y_0-1\ge 0$, and thus $i_2$ satisfies the inequality in (ii).

Conversely,  if $i_2$ satisfies the inequality in (ii), then $\alpha'\in S$.
In fact, we have
$1\le w:=1+i_2 \le k$ and 
$\alpha'_{n-1}-\alpha'_0=\alpha_{n-1}-\alpha_0<N$, and Property I$^\prime$
then  follows. 

For Property II$^\prime$, by the definition of $j_{v+1}$,
we have $\Delta_{u}\ge \Delta_{j_{v+1}}$ for any $j_v<u\in J$, and
hence for any $j_v<u\le n-1$ such that $x_u>x_{j_v}$. 
Thus $ru-x_u\ge \Delta_{j_{v+1}}$. It follows that
$$ru\ge \Delta_{j_{v+1}}+x_u > \Delta_{j_{v+1}}+x_{j_v} =
\Delta_{j_{v+1}}+r j_v-\Delta_{j_{v}}  $$  
and 
$$u-j_v>(\Delta_{j_{v+1}}-\Delta_{j_{v}} ) /r\ge w/r,
\ \text{i.e.,}\ u\ge q(w)+j_v.$$ 
Hence $\alpha'_{q(w)-1}=x_{q(w)-1+j_v}-x_{j_v}+w=w$.

Finally for Property III$^\prime$,  
from the above argument we see that $\alpha'_{u}\le ru$ for
$u=j_{v+1}-j_v,n-{j_v}$. 
Further, we have $$\alpha'_{n-j_v+u}=x_u-x_{j_v}+1+N+i_2 \le
ru-x_{j_v}+1+rn-k+k-2-\Delta_{j_v}<r(n-{j_v}+u)$$  
for all $u\in[{j_v}-1]$. 
For $j_v+1\le u\le n-1$ such that $x_u>x_{j_v}$ and $u\in J$, 
we have $\Delta_{u}\ge \Delta_{j_{v+1}}$ by the definition of $j_{v+1}$, 
and therefore 
$$\alpha'_{u-j_v}=x_{u}-x_{j_v}+w
\le (ru-\Delta_{j_{v+1}})-x_{j_v}+(\Delta_{j_{v+1}}-\Delta_{j_{v}}) =
r(u-j_v).$$ 

Hence $\alpha'\in S$, as desired.
\qed

\medskip
This completes the proof.
\qed

\medskip
\textsc{Note.} We have been unable to find another proof or a satisfactory $q$-analogue
of Theorem~\ref{thm:G}, generalizing Theorem~\ref{thm:q}.
In particular, it would be interesting to find a geometric proof of  Proposition \ref{prop:rotation} in terms of paths.

\section{The $r$-parking function basis}
Equation~\eqref{eq:cprk} and other considerations suggest
looking at products of the symmetric functions $F^{(r,k)}_n$ for various
values of $n$. Thus for any partition $\lambda$ define
  $$ F^{(r,k)}_\lambda =
     F^{(r,k)}_{\lambda_1}F^{(r,k)}_{\lambda_2}\cdots, $$ 
where $F_0=1$, and $ F^{(r)}_\lambda = F^{(r,1)}_\lambda$. 

Recall that $\Lambda$ denotes the ring of all symmetric functions that
can be written as an integer linear combination of the monomial
symmetric functions $m_\lambda$ (or equivalently, $s_\lambda$,
$h_\lambda$, or $e_\lambda$). 

\begin{prop}
Fix $r\geq 1$. Then the symmetric functions $F_\lambda^{(r)}$, where
$\lambda$ ranges over all partitions of all $n\geq 0$, form an integral
basis for the ring $\Lambda$.
\end{prop}

\proof
We need to show that for each $n$, the set $\{F^{(r)}_\lambda\st
\lambda\vdash n\}$ is an integral basis for the (additive) group
$\Lambda^n$ of all homogeneous symmetric functions of degree $n$
contained in $\Lambda$.  Let $\lambda^1,\lambda^2,\dots$ be any
ordering of the partitions of $n$ that is compatible with refinement,
that is, if $\lambda^i$ is a refinement of $\lambda^j$ then $i\leq
j$. Now $F^{(r)}_n=h_n + \cdots\in\Lambda^n$. Hence
$F_\lambda^{(r)}=h_\lambda\, +$ terms involving $h_\mu$ where $\mu$
refines $\lambda$. Hence the transition matrix for expressing the
$F^{(r)}_\lambda$'s in terms of the $h_\lambda$'s is lower triangular
with 1's on the main diagonal. Since the $h_\lambda$'s form an
integral basis, the same is true of the $F^{(r)}_\lambda$'s.
\qed

\medskip
Now that for each $r\geq 1$ we have this ``parking function basis''
$\{F^{(r)}_\lambda\}$, we can ask about its expansion in terms of
other bases and vice versa. If we restrict ourselves to the six
``standard'' bases (where the power sums $p_\lambda$ are a basis over
$\qq$ but not $\zz$), we thus have twelve transition matrices to
consider. We can also ask about various scalar products such as
$\langle F_\lambda^{(r)},F_\mu^{(r)}\rangle$. Moreover,
we could also consider the basis $\{\pflad\}$ dual to
$\{F^{(r)}_\lambda\}$, i.e.,
  $$ \langle F^{(r)}_\lambda, \tilde{F}^{(r)}_\mu\rangle =
     \delta_{\lambda\mu}. $$ 
However, these dual bases will not yield any new coefficients since
the dual basis to a standard basis is also a standard basis (up to a
normalizing factor in the case of $p_\lambda$).  We have not
systematically investigated these problems. Some miscellaneous results
are below.

We first consider scalar products $\langle
F_\mu^{(r,k)},F_\lambda^{(r,k)}\rangle$. We can give an explicit
formula when $\mu=(n)$. In fact, we can give a more general result
where $F_\lambda^{(r,k)}$ is replaced with a ``mixed'' product.

\begin{thm} \label{thm:pfsp}
Let $\lambda\vdash n$, and let $r,r_1,r_2,\dots$ be positive
integers. Let $k,k_1,$ $k_2,\dots$ be integers or even indeterminates.
Then 
  $$ \left\langle \fnrx, \prod_{i\geq 1} F_{\lambda_i}^{r_i,k_i} \right\rangle
    = \frac{k}{rn+k}\prod_{i\geq 1}\frac{k_i}{r_i\lambda_i+k_i}
    \binom{(rn+k)(r_i\lambda_i+k_i)+\lambda_i-1}{\lambda_i}. $$ 
\end{thm}

\proof
\emph{First proof.}
If $\lambda=(\lambda_1,\lambda_2,\dots)$ then write $[t^\lambda]$ for
the operator that takes the coefficient of
$t_1^{\lambda_1}t_2^{\lambda_2}\cdots$. By equation~\eqref{eq:[t^n]H} we
have
\begin{align*}
\left\langle\!\! \fnrx\!, \prod_{i\geq 1}\! F_{\lambda_i}^{r_i,k_i} \!\!\right\rangle\!
=\! \frac{k}{rn\!+\!k}\! \prod_{i\geq 1}\! \frac{k_i}{\lambda_i\!+\!k_i}[t^\lambda]
\!   \left\langle\! H(1)^{rn+k}\!,H(t_1)^{r_1\lambda_1+k_1} \!
    H(t_2)^{r_2\lambda_2+k_2}\!\cdots\right\rangle\!.
\end{align*}
Writing $H(u)^b=\prod_i (1-x_iu)^b$, taking logarithms,
expanding in terms of the power sums $p_k$, and then exponentiating,
we get the well-known result
   $$ H(u)^b = \sum_\mu  z_\mu^{-1}b^{\ell(\mu)}p_\mu
       u^{|\mu|}, $$ 
where $\mu$ ranges over all partitions of all integers $j\geq 0$. (For
the case $b=1$, see \cite[(7.22)]{ec2}.) 

Since $\langle p_\lambda,p_\mu\rangle = z_\lambda\delta_{\lambda\mu}$,
we get
\begin{align} 
& \left\langle \fnrx, \prod_{i\geq 1} F_{\lambda_i}^{(r_i,k_i)}
 \right\rangle \nonumber \\
=&\, \frac{k}{rn\!+\!k}\prod_{i\geq 1} \frac{k_i}{r_i\lambda_i\!+\!k_i}
[t^\lambda]\! \left\langle \sum_{\mu\vdash n} z_\mu^{-1}(rn+k)^{\ell(\mu)}  p_\mu, \prod_{i\geq 1} \sum_{\nu\vdash\lambda_i}\! z_\nu^{-1}
   (r_i\lambda_i+k_i)^{\ell(\nu)}t_i^{|\nu|}p_\nu\!\!
   \right\rangle\nonumber\\  
= &\, \frac{k}{rn\!+\!k}\prod_{i\geq 1} \frac{k_i}{r_i\lambda_i\!+\!k_i}\cdot \prod_{i\geq 1} \sum_{\nu\vdash \lambda_i} z_\nu^{-1}(rn+k)^{\ell(\nu)} (r_i\lambda_i+k_i)^{\ell(\nu)}. \label{eq:scpr} 
\end{align}
Now in general (equivalent for instance to \cite[Prop.~1.3.7]{ec1}),
  $$ \sum_{\nu\vdash m}z_\nu^{-1}u^{\ell(\nu)} =
         \binom{u+m-1}{m}. $$
Hence the proof follows immediately from equation~\eqref{eq:scpr}.

\medskip
\emph{Second proof.}
From equation~\eqref{eq:[t^n]H} we see that
\begin{align*}
 \left\langle \fnrx, \prod_{i\geq 1} F_{\lambda_i}^{r_i,k_i} \right\rangle 
     = &\ \frac{k}{rn+k}\prod_{i\geq 1}
  \frac{k_i}{r_i\lambda_i+k_i}\cdot
    \left\langle \sum_{a_1+\cdots+a_{rn+k}=n}h_{a_1}\cdots h_{a_{rn+k}},\right.\\
   &\quad \quad \quad \quad \quad \left.\prod_{i\geq 1} 
     \sum_{b_{i,1}+\cdots+b_{i,r_in+k_i}=\lambda_i}h_{b_{i,1}}\cdots
     h_{b_{i,r_in+k_i}}  \right\rangle, 
\end{align*}
where $a_i,b_{i,j}\geq 0$. Let
   $$ Z=\frac{k}{rn+k}\prod_{i\geq 1}\frac{k_i}{r_i\lambda_i+k_i}. $$
Now $\langle h_\lambda,h_\mu\rangle$ is equal to the number of
matrices $(a_{ij})_{i,j\geq 1}$ of nonnegative integers with row sum
vector $\lambda$ and column sum vector $\mu$ \cite[(7.31)]{ec2}.
Hence $\frac{1}{Z}\left\langle \fnrx, \prod_i F_{\lambda_i}^{r_i,k_i}
\right\rangle$ is equal to the total number of $(rn+j)\times
\big(\sum_i (r_i+n+k_i)\big)$ matrices of nonnegative integers whose
entries sum to $n$, such that the first $r_1\lambda_1+k_1$ columns sum
to $\lambda_1$, the next $r_2\lambda_2+k_2$ columns sum to
$\lambda_2$, etc. Since $\sum \lambda_i=n$, if the conditions on the
columns is satisfied then the entries will automatically sum to $n$.
By elementary and well-known reasoning, the number of ways to write
$\lambda_i$ as an ordered sum of $(rn+k)(r_in+k_i)$ nonnegative
integers is $\binom{(rn+k)(r_i\lambda_i+k_i)+\lambda_i-1}{\lambda_i}$,
and the proof follows.
\qed

\medskip
We now consider the expansion of the symmetric functions $p_\lambda$,
$h_\lambda$, and $e_\lambda$ in terms of the basis $\fnr$ (for fixed
$r$, which we may even regard as an indeterminate). 

\begin{prop} \label{prop:pn}
For $n\geq 1$ we have 
	\begin{align*}
	     F_n^{(r,-rn-1)} \  = &\ (-1)^n(rn+1)e_n\\
         F_n^{(r,-rn)} \ = &\ -rp_n\\
         F_n^{(r,-rn+1)} \ = &\ (1-rn)h_n.
    \end{align*} 
\end{prop}

\proof
Putting $k=-rn-1$ in equation~\eqref{eq:plambda} gives
$(-1)^n(rn+1)\sum_{\lambda\vdash
  n}z_\lambda^{-1}$ $\cdot (-1)^{n-\ell(\lambda)}p_\lambda$. It is well-known
that this sum is just $e_n$, and the proof of the first equation
follows. (We could also substitute $k=-rn-1$ in
equation~\eqref{eq:elambda} and simplify.) The other two equations are similar. 
\qed

\medskip
Now by Proposition~\ref{prop:pn} we have (writing $d_i=d_i(\lambda)$) 
\begin{align*}  (-1)^n(rn+1)e_n 
= &\ F_n^{(r,-rn-1)}\\ 
= &\ 
    [t^n]\left(\sum_{i\geq 0} F_i^{(r)}t^i\right)^{-rn-1}\\
     = &\ [t^n]\sum_{j\geq 0}(-1)^j\binom{rn+j}{j}\left( \sum_{i\geq 1}
       F_i^{(r)}t^i\right)^j\\ 
       = &\
    \sum_{a_1+\cdots a_j=n}(-1)^j\binom{rn+j}{j}F_{a_1}^{(r)}\cdots
    F_{a_j}^{(r)}\\ 
    = &\ 
   \sum_{\lambda\vdash n}(-1)^{\ell(\lambda)}
    \binom{rn+\ell(\lambda)}{d_1,d_2,\dots,rn}F_\lambda^{(r)},
\end{align*}
where the penultimate sum is over all $2^{n-1}$ compositions of
$n$. We have therefore expressed $e_n$ as a linear combination of
$F_\lambda^{(r)}$'s. In exactly the same way we obtain
\begin{align*} 
-r p_n   
= &\  \sum_{\lambda\vdash n} (-1)^{\ell(\lambda)}
   \binom{rn+\ell(\lambda)-1}{d_1,d_2,\dots,rn-1}F_\lambda^{(r)}\\
  -(rn-1)h_n  
  = &\ \sum_{\lambda\vdash n} (-1)^{\ell(\lambda)}
   \binom{rn+\ell(\lambda)-2}{d_1,d_2,\dots,rn-2}F_\lambda^{(r)}.
\end{align*}
(For $r=n=1$, the last equation becomes $0=0$, but it is clear that
$h_1= F_1^{(r)}$.)  Since $\{e_\mu\}, \{p_\mu\}, \{h_\mu\}$ and
$\{F_\lambda^{(r)}\}$ are multiplicative bases, we have in principle
expressed each $e_\mu, p_\mu$, and $h_\mu$ as a linear combination of
$F_\lambda^{(r)}$'s.  We leave open, however, whether there is some
more elegant form of these expansions, e.g., a simple combinatorial
interpretation of the coefficients. 

Similarly, since Theorem~\ref{thm:cprk} in the case $k=1$ gives the
expansion of $F_n^{(r)}$ in terms of the multiplicative bases $p_\mu$,
$h_\mu$, and $e_\mu$, we in principle also have an expansion of
$F^{(r)}_\lambda$ in terms of these bases, but perhaps a better
description is available. We cannot expect a simple product formula
for the coefficients in general since for instance the coefficient of
$p_3p_6$ in the power sum expansion of $F^{(1)}_{(3,2,1,1,1,1)}$ is
equal to $2\cdot 7\cdot 157/3$.

\section*{Acknowledgements}

The authors are grateful to the referees and the editor for their very careful and helpful comments.




\end{document}